\documentclass{amsart}

\usepackage[T1]{fontenc}
\usepackage[utf8]{inputenc}
\usepackage{lmodern}
\usepackage{mathtools}
\usepackage{microtype}
\usepackage[sort&compress,numbers,square]{natbib}
\usepackage{hyperref}
\hypersetup{
  pdfauthor={Zijia Li, Josef Schicho, Hans-Peter Schröcker},  pdftitle={The Rational Motion of Minimal Dual Quaternion Degree With Prescribed Trajectory},  pdfkeywords={rational motion, dual quaternion, minimal degree, trajectory generation, circularity},    hidelinks=true,}

\newtheorem{theorem}{Theorem}
\newtheorem{lemma}{Lemma}
\newtheorem{corollary}{Corollary}
\theoremstyle{definition}
\newtheorem{definition}{Definition}
\theoremstyle{remark}
\newtheorem{remark}{Remark}
\newtheorem{example}{Example}

\title[The Rational Motion of Minimal Dual Quaternion Degree \ldots]{The Rational Motion of Minimal Dual Quaternion Degree With Prescribed Trajectory}
\date{\today}

\author{Zijia Li}
\address[Zijia Li]{Johann Radon Institute for Computational and Applied Mathematics, Austrian Academy of Science, Altenberger Str.~69, 4040 Linz, Austria}
\urladdr{https://people.ricam.oeaw.ac.at/z.li/}
\email{zijia.li@oeaw.ac.at}

\author{Josef Schicho}
\address[Josef Schicho]{Research Institute for Symbolic Computation,
  Johannes Kepler University Linz, Schloss Hagenberg, 4232 Hagenberg, Austria}
\urladdr{http://www.risc.jku.at/people/jschicho/}
\email{josef.schicho@risc.jku.at} 
\author{Hans-Peter Schröcker}
\address[Hans-Peter Schröcker]{Unit Geometry and CAD, University of Innsbruck, Technikerstr.~13, 6020 Innsbruck, Austria}
\urladdr{http://geometrie.uibk.ac.at/schroecker/}
\email{Hans-Peter Schröcker}


\newcommand{\C}{\mathbb{C}}
\newcommand{\D}{\mathbb{D}}
\renewcommand{\H}{\mathbb{H}}
\renewcommand{\P}{\mathbb{P}}
\newcommand{\R}{\mathbb{R}}
\renewcommand{\DH}{\D\H}
\newcommand{\qi}{\mathbf{i}}
\newcommand{\qj}{\mathbf{j}}
\newcommand{\qk}{\mathbf{k}}
\newcommand{\eps}{\varepsilon}
\newcommand{\linspan}[1]{\langle #1 \rangle}
\newcommand{\cj}[1]{\overline{#1}}

\newcommand{\SE}[1][3]{\mathrm{SE}(#1)}
\newcommand{\SQ}{S}
\newcommand{\peq}{\equiv}
\DeclareMathOperator{\mrpf}{mrpf}

\begin{document}

\begin{abstract}
  We give a constructive proof for the existence of a unique rational
  motion of minimal degree in the dual quaternion model of Euclidean
  displacements with a given rational parametric curve as
  trajectory. The minimal motion degree equals the trajectory's degree
  minus its circularity. Hence, it is lower than the degree of a
  trivial curvilinear translation for circular curves.
\end{abstract}

\maketitle

\section{Introduction}
\label{sec:introduction}

A rational motion is a motion with only rational trajectories. In the
dual quaternion model of $\SE$, the group of rigid body displacements,
\cite[Ch.~9]{selig05} it is described by a rational curve on the Study
quadric \cite{juettler93}. In this article we construct a rational
motion of minimal degree in the dual quaternion model with a given
rational curve as trajectory, and we show that this motion is unique
up to coordinate changes. This is an interesting result in its own
right but it also has a certain potential for applications in computer
graphics, computer aided design or mechanism science.

Usually, one defines the degree of a rational motion as the maximal
degree of a trajectory \cite{juettler93}. With this concept of motion
degree, our problem becomes trivial as the curvilinear translation
along the curve is already minimal. As we shall see, it is also
minimal with respect to the dual quaternion degree if the prescribed
trajectory is generic. The situation changes, however, if the
trajectory is \emph{circular}, that is, it intersects the absolute
circle at infinity. In this case, the minimal achievable degree in the
dual quaternion model is the curve degree minus half the number of
conjugate complex intersection points with the absolute circle at
infinity (the curve's \emph{circularity}).

We will see that twice the circularity of a trajectory equals the
trajectory degree minus the degree defect in the spherical component
of the minimal motion. This leads to the rather strange observation
that generic rational motions (without spherical degree defect) have
very special (entirely circular) trajectories. Conversely, the minimal
motion to generic (non-circular) curves are curvilinear translations
which are special in the sense that their spherical degree defect is
maximal.

We continue this article with an introduction to the dual quaternion
model of rigid body displacements in
\autoref{sec:dual-quaternion-model}. There we also introduce motion
polynomials and their relation to rational motions. Our results are
formulated and proved in \autoref{sec:rational-motions}. The proof of
the central result (\autoref{th:2}) is constructive and can be used to
actually compute the minimal rational motion by a variant of the
Euclidean algorithm. We illustrate this procedure by two examples.

\section{The dual quaternion model of rigid body displacements}
\label{sec:dual-quaternion-model}

In this article, we work in the dual quaternion model of the group of
rigid body displacements. This model requires a minimal number of
parameters while retaining a bilinear composition law. Moreover, it
provides a rich algebraic and geometric structure. It is, for example,
possible to use a variant of the Euclidean algorithm for computing the
greatest common divisor (gcd) of two polynomials. This section
presents the necessary theoretical background on dual quaternions.

\subsection{Dual quaternions}
\label{sec:dual-quaternions}

The set $\DH$ of dual quaternions is an eight-dimensional associative
algebra over the real numbers. It is generated by the base elements
\begin{equation*}
  1,\quad
  \qi,\quad
  \qj,\quad
  \qk,\quad
  \eps,\quad
  \eps\qi,\quad
  \eps\qj,\quad
  \eps\qk
\end{equation*}
and the non-commutative multiplication is determined by the relations
\begin{equation*}
  \qi^2 = \qj^2 = \qk^2 = \qi\qj\qk = -1,\quad
  \eps^2 = 0,\quad
  \qi\eps = \eps\qi,\quad
  \qj\eps = \eps\qj,\quad
  \qk\eps = \eps\qk.
\end{equation*}
As important sub-algebras, the algebra of dual quaternions contains
the real numbers $\R = \linspan{1}$, the complex numbers
$\C = \linspan{1, \qi}$, the dual numbers $\D = \linspan{1, \eps}$,
and the quaternions $\H = \linspan{1,\qi,\qj,\qk}$ (angled brackets
denote a linear span). A dual quaternion may be written as
$h = p + \eps q$ where $p$, $q \in \H$ are quaternions. The conjugate
dual quaternion is $\cj{h} = \cj{p} + \eps\cj{q}$ and conjugation of
quaternions is done by multiplying the coefficients of $\qi$, $\qj$,
and $\qk$ with $-1$. It can readily be verified that the dual quaternion
norm, defined as $\Vert h \Vert = h\cj{h}$, equals
$p\cj{p} + \eps (p\cj{q} + q\cj{p})$. It is a dual number. The
non-invertible dual quaternions $h = p + \eps q$ are precisely those
with vanishing primal part $p = 0$.

An important application of dual quaternions is the modelling of rigid
body displacements. The group of dual quaternions of unit norm modulo
$\{\pm 1\}$ is isomorphic to $\SE$, the group of rigid body
displacements. A unit dual quaternion $h = p + \eps q$ acts on a point
$x$ in the three dimensional real vector space
$\linspan{\qi, \qj, \qk}$ according to
\begin{equation}
  \label{eq:1}
  x \mapsto px\cj{p} + p\cj{q} - q\cj{p} = px\cj{p} +2p\cj{q}.
\end{equation}
Note that $-q\cj{p} = p\cj{q}$ because of the unit norm condition.  It
is convenient and customary to projectivise the space $\DH$ of dual
quaternions, thus arriving at $\P^7$, the real projective space of
dimension seven. Then, the unit norm condition can be relaxed to the
non-vanishing of $p\cj{p}$ and the vanishing of $p\cj{q} + q\cj{p}$.
In a geometric language, $\SE$ is isomorphic to the points of a
quadric $\SQ \subset \P^7$, defined by $p\cj{q} + q\cj{p} = 0$, minus
the points of a three-dimensional space, defined by $p = 0$. The
quadric $\SQ$ is called the \emph{Study quadric.} In this setting, the
map \eqref{eq:1} becomes
\begin{equation*}
  x \mapsto \frac{px\cj{p} + p\cj{q} - q\cj{p}}{p\cj{p}}
          = \frac{px\cj{p} + 2p\cj{q}}{p\cj{p}}.
\end{equation*}
The action of $h = p + \eps q$ with $p \neq 0$,
$p\cj{q} + q\cj{p} = 0$ can be extended to real projective three-space
$\P^3$, modelled as projective space over
$\linspan{1, \qi, \qj, \qk}$. The point $x$ represented by
$x_0 + x_1\qi + x_2\qj + x_3\qk$ is mapped according to
\begin{equation*}
  x \mapsto x_0p\cj{p} + p(x_1\qi + x_2\qj + x_3\qk)\cj{p} + 2x_0p\cj{q} =
            px\cj{p} + 2x_0p\cj{q}.
\end{equation*}
This is a convenient representation for studying rational curves as
trajectories of rational motions.

\subsection{Rational motions and motion polynomials}
\label{sec:motion-polynomials}

In the projective setting, a rational motion is simply a curve in the
Study quadric $\SQ$ that admits a parameterisation by a polynomial
\begin{equation}
  \label{eq:2}
  C = \sum_{i=0}^n c_it^i
\end{equation}
with dual quaternion coefficients $c_0,\ldots,c_n \in \DH$. The
non-commutativity of $\DH$ necessitates some rules concerning notation
and multiplication: Polynomial multiplication is defined by the
convention that the indeterminate commutes with all coefficients, the
ring of these polynomials in the indeterminate $t$ is denoted by
$\DH[t]$, the subring of polynomials with quaternion coefficients is
$\H[t]$. We always write coefficients to the left of the indeterminate
$t$. This convention is sometimes captured in the term ``left
polynomial'' but we will just speak of a ``polynomial''.

Evaluating $C$ at different values $t \in \R$ gives points of a
rational curve in $\P^7$. We also define
$C(\infty) \coloneqq c_n = \lim_{t \to \infty} t^{-n} C(t)$ in order
to obtain the complete curve as image of
$\P^1 \coloneqq \R \cup \{ \infty \}$ under the map $t \mapsto
C(t)$. A similar convention is used for rational curves in~$\P^3$.

The conjugate to the polynomial \eqref{eq:2} is
\begin{equation*}
  \cj{C} \coloneqq \sum_{i=0}^{n} \cj{c_i}t^i.
\end{equation*}
It can readily be verified that the \emph{norm polynomial} $C\cj{C}$
has dual numbers as coefficients. If $C\cj{C}$ has even \emph{real}
coefficients and the leading coefficient $c_n$ is invertible, we call
$C$ a \emph{motion polynomial.} This is motivated by the observation
that the curve parameterised by a motion polynomial is contained in
the Study quadric, that is, it parameterises indeed a rigid body
motion. Writing $C = P + \eps Q$ with $P$, $Q \in \H[t]$, motion
polynomials are characterised by the vanishing of 
$P\cj{Q}+Q\cj{P}$.

The trajectory of the point $x = x_0 + x_1\qj + x_2\qj + x_3\qk \in \P^3$ is
\begin{equation}
  \label{eq:3}
  t \mapsto Px\cj{P} + 2x_0P\cj{Q},\quad
  t \in \P^1.
\end{equation}
This is a curve parameterised by polynomial functions in projective
coordinates. It is also called a \emph{rational curve} because it is
always possible to clear denominators of rational functions. All
motions in $\SE$ with only rational trajectories can be parameterised
by motion polynomials \cite{juettler93}.

\section{Rational curves as trajectories of rational motions}
\label{sec:rational-motions}

A rational parameterised equation $x = x_0 + x_1\qi + x_2\qj + x_3\qk$
with $x_0$, $x_1$, $x_2$, $x_3 \in \R[t]$ is called \emph{reduced} if
the greatest common divisor $g$ of $x_0$, $x_1$, $x_2$, $x_3$ has
degree zero. The degree of $x$ is defined as the maximal degree of
$x_0$, $x_1$, $x_2$, $x_3$. If $x$ is not reduced, we may divide it by
$g$ to obtain an equivalent parameterised equation which describes the
same rational curve, but possibly with fewer parameterisation
singularities.

Given a reduced parameterised equation $x$ of degree $d$, it is our
ultimate aim to find a motion polynomial $C = P + \eps Q \in \DH[t]$
of minimal degree such that the trajectory of one point $p$ is
parameterised by $x$. An obvious example of a motion polynomial with
trajectory $x$ is the curvilinear translation along $x$ and it will
turn out that this is already the solution to our problem in generic
cases. However, for trajectories which are non-generic, in a sense to
be made precise below, a lesser degree can be achieved.

\subsection{Circularity of trajectories}
\label{sec:circularity}

A motion polynomial $C = P + \eps Q$ is called \emph{reduced,} if both
$P$ and $Q$ have no common real factor. The real factor of the primal
part $P$ with maximal degree is uniquely defined up to a constant
scalar factor. We call it \emph{maximal real polynomial factor} and
abbreviate it by ``mrpf''. It accounts for a difference in degrees
between the rational motion, parameterized by $C = P + \eps Q$, and
its spherical component, parameterised by $P$. Hence, we also call the
degree of the maximal real polynomial factor of $P$ the
\emph{spherical degree defect} of~$C$.

We only consider reduced rational motions but do not exclude motions
with positive spherical degree defect. It will turn out that the
spherical degree defect and circularity of trajectories are closely
related:

\begin{definition}
  \label{def:1}
  The \emph{circularity} of a rational curve with reduced rational
  parameterised equation $x = x_0 + x_1\qi + x_2\qj + x_3 \qk$ is
  defined as
  \begin{equation*}
    \frac{1}{2} \deg \gcd (x_0,x_1^2 + x_2^2 + x_3^2).
  \end{equation*}
  The curve is called \emph{entirely circular} if it is of maximal
  circularity $\frac{1}{2}\deg x$.
\end{definition}

Geometrically, circularity is half the number of (conjugate complex)
intersection points of $x$ and the absolute circle at infinity,
counted with their respective algebraic multiplicities. Hence, it is
always an integer (this also follows algebraically from the fact that
$x_0$, $x_1$, $x_2$, $x_3$ are real and prime) and does not depend on
the chosen parameterisation, as long as it is reduced.

\begin{theorem}
  \label{th:1}
  Let $C$ be a reduced motion polynomial of degree $n$ and spherical
  degree defect $m$.  Then a trajectory of $C$ is of degree
  $d \leq 2n - m$ and circularity $c \geq (d-m)/2$.
\end{theorem}

\begin{remark}
  \label{rem:1}
  In \autoref{th:1}, the strictly less and strictly greater cases can
  occur. One example of this is the general Darboux motion
  \cite{SL,7R} where $n = 3$, $m = d = 2$, $c = 0$ and hence
  $2 = d < 2n-m = 4$ (but $0 = c = (d-m)/2$). Another example is the
  circular translation motion \cite{SL,7R} with $n = m = d = 2$,
  $c = 1$. Here $2 = d = 2n-m$ but $1 = c > (d-m)/2 = 0$.
\end{remark}

\begin{lemma}
  \label{lem:1}
  Assume $A,B\in\H[t]$. If a monic quadratic irreducible real
  polynomial $Q$ is a factor of the product $AB$, then $Q$ either
  divides $A$ or $Q$ divides $B$ or there is a unique quaternion
  $q\in\H$ and two quaternion polynomials $A_1$ and $B_1$ such that
  $Q=(t-q)(t-\cj{q})$, $A=A_1(t-q)$ and $B=(t-\cj{q})B_1$.
\end{lemma}
\begin{proof}
  It is sufficient to prove the third case under the assumption
  $\gcd(Q,A_0)=\gcd(Q,B_0)=1$ where $A_0 \coloneqq \mrpf A$,
  $B_0 \coloneqq \mrpf B$. There exist uniquely determined quaternions
  $q_1,\ldots,q_r$, $q'_1,\ldots,q'_s$ such that
  \begin{itemize}
  \item $Q = (t-q_i)(t-\cj{q_i}) = (t-q'_j)(t-\cj{q'_j})$ for
    $i \in \{1,\ldots,r\}$ and $j \in \{1,\ldots,s\}$,
  \item $A = A_0Q_A$, $B = B_0Q_B$ where
    $Q_A \coloneqq (t-q_r)\cdots(t-q_2)(t-q_1)$,
    $Q_B \coloneqq (t-q'_1)(t-q'_2)\cdots(t-q'_s)$, and
  \item $\gcd(A_0\cj{A_0},Q) = \gcd(B_0\cj{B_0},Q) = 1$.
  \end{itemize}
  This follows from the quaternion version of the factorisation
  theorem for motion polynomials \citep[Theorem~1]{hegedus13}.

  As $Q$ divides $AB$, we get that $Q$ divides
  $\cj{A_0}AAB\cj{B_0} = \cj{A_0}A_0Q_AQ_BB_0\cj{B_0}$. But then $Q$
  also divides $Q_AQ_B$. This can happen only if two neighbouring
  linear factors are conjugate and, because of $\gcd(Q,\mrpf A) =
  \gcd(Q,\mrpf B) = 1$, implies $q_1 = \cj{q'_1}$ whence
  the claim follows with $q \coloneqq q_1 = \cj{q'_1}$.
\end{proof}

\begin{proof}[Proof of \autoref{th:1}]
  Let $T \coloneqq 2x_0+\eps(x_1\qi+x_2\qj+x_3\qk)$ be the
  translation which moves a point $x= x_0 + x_1\qi + x_2\qj + x_3\qk$
  to the origin $1$ and let $\tilde{C}=CT^{-1}$. Then the orbit of $x$
  with respect to $C$ is equal to the orbit of the origin $1$ with
  respect to $\tilde{C}$.  Because we have $\deg(\tilde{C})=\deg(C)$
  and the spherical degree defect of~$C$ equals the spherical degree
  defect of~$\tilde{C}$, it is sufficient to prove the statement for
  $\tilde{C}$ and the origin $1$ (instead of $C$ and $x$). From now
  on, we use the symbol $C$ to abbreviate $\tilde{C}$. Let
  $C= P + \eps Q$, set $G \coloneqq \mrpf P$, $m \coloneqq \deg G$,
  $P' \coloneqq P/G$ and use \eqref{eq:3} to compute a parametric
  equation $y$ for the trajectory of the origin $1$:
  \begin{equation*}
      y 
      = P\cj{P} + 2P\cj{Q} 
      = G^2P'\cj{P'} + 2GP'\cj{Q} 
      \peq GP'\cj{P'} +  2P'\cj{Q},
  \end{equation*}
  where we write ``$\peq$'' for equality in the projective sense, modulo
  scalar (or real polynomial) multiplication.  The trajectory's degree
  is not larger than $2n-m$.

  With $N \coloneqq \gcd(GP'\cj{P'},\mrpf(P'\cj{Q}))$ and
  $r \coloneqq \deg N$, the degree of the trajectory $y$ is
  $d = 2n-m-r$ because we have $y\peq (GP'\cj{P'}+2P'\cj{Q})/{N}$.  If
  $P'\cj{P'}/N$ and $Q\cj{Q}/N$ are polynomials, the circularity is
  half the degree of
  \begin{equation*}
   \gcd\Bigl(\frac{GP'\cj{P'}}{N}, \frac{P'\cj{P'}Q\cj{Q}}{N^2}\Bigr) =
   \frac{P'\cj{P'}}{N}\gcd\Bigl(G,\frac{Q\cj{Q}}{N}\Bigr).
  \end{equation*}
  This would then imply that the circularity is not less than
  $n-m-r/2$, which equals $(d-m)/2$ for $d=2n-m-r$. Thus, we have to
  show that $N$ divides both, $P'\cj{P'}$ and $Q\cj{Q}$.
  
  Moreover, it is sufficient to prove the claim only for a reduced
  $C_0 \coloneqq P_0+\eps Q_0$ as the claim in the non-reduced case
  follows easily. Set
  $N_0 \coloneqq \gcd(G_0P'_0\cj{P'_0}, \mrpf P'_0\cj{Q_0})$ where
  $G_0 \coloneqq \mrpf P_0$ and $P'_0 \coloneqq P_0/G_0$. It is a
  useful fact that $N_0$ has no real linear factor because such a
  factor would also divide $G_0$ (because $P'_0\in\H[t]$ has no real
  factor and hence $P'_0\cj{P'_0}$ has no real linear factor either)
  and $Q_0$ (because it divides $P'_0\cj{Q_0}$ and $P'_0$ has no real
  factor). But this contradicts the reducedness of $C_0$.  Now we
  proceed by induction on the degree $s$ of $N_0$. Note that $s$ is
  zero or an even positive integer.
  
  In case of $s=0$, $N_0$ is a real constant and the claim is clear.

  Let $s \geq 2$. Then there is a monic quadratic irreducible real
  polynomial $M_1$ which divides $N_0$. Set $N_1 \coloneqq N_0/M_1$.
  We claim that there always exists a quaternion $q$ and two
  quaternion polynomials $P'_1$ and $Q_1$ such that
  $M_1=(t-q)(t-\cj{q})$, $P'_0=P'_1(t-q)$ and $Q_0=Q_1(t-q)$.  We
  handle the proof of this claim by distinguishing three cases:
  \begin{enumerate}
  \item $M_1$ divides $G_0$. Then $M_1$ can not divide $\mrpf Q_0$
    because of the reducedness of $C$.  Since $M_1$ is a divisor of
    $N_0$ which divides $P'_0\cj{Q_0}$, $M_1$ must divide
    $P'_0\cj{Q_0}$.  Since $P'_0$ has no real polynomial factor, by
    \autoref{lem:1}, there is a unique quaternion $q$ and two polynomials
    $P'_1$ and $Q_1$ such that $M_1=(t-q)(t-\cj{q})$, $P'_0=P'_1(t-q)$
    and $Q_0=Q_1(t-q)$.
  \item $M_1$ does not divide $G_0$ but divides $\mrpf Q_0$.  Then
    $M_1$ must divide $P'_0\cj{P'_0}$.  Since $P'_0$ has no real
    polynomial factor, by \autoref{lem:1}, there is a unique
    quaternion $q$ such that $M_1=(t-q)(t-\cj{q})$, $P'_0=P'_1(t-q)$.
        Then we set $Q_1 \coloneqq Q_0(t-\cj{q})/M_1$ which is a
    polynomial and satisfies $Q_0 = Q_1(t-q)$.
  \item $M_1$ divides neither $G_0$ nor $\mrpf Q_0$. By
    \autoref{lem:1}, there is a unique quaternion $q$ and two
    polynomials $P'_1$ and $Q_1$ such that $M_1=(t-q)(t-\cj{q})$,
    $P'_0=P'_1(t-q)$ and $Q_0=Q_1(t-q)$.
  \end{enumerate}
  In all three cases, we have a new reduced motion polynomial
  $C_1 \coloneqq G_0P'_1+\eps Q_1$, such that
  \begin{equation*}
    P'_0=P'_1(t-q),\quad
    Q_0=Q_1(t-q),\quad
    \gcd(G_0P'_1\cj{P'_1}, \mrpf(P'_1\cj{Q_1})) = N_0/M_1 \eqqcolon N_1.
  \end{equation*}
  By induction hypothesis, $N_1$ divides $P'_1\cj{P'_1}$ and
  $Q_1\cj{Q_1}$.  By the derivation of $P'_1$ and $Q_1$, namely,
  $P'_0=P'_1(t-q)$ and $Q_0=Q_1(t-q)$, we have that $N_0$ divides
  $P'_0\cj{P'_0}$ and $Q_0\cj{Q_0}$.  This concludes the induction
  proof and also the proof of the theorem.
\end{proof}
   
\begin{remark}
  Calling a rational motion \emph{generic} if the primal part has no
  real factors one interpretation of \autoref{th:1} is as follows: A
  trajectory of a generic rational motion can be entirely circular
  (actually, only few exceptions are known). A similar statement for
  algebraic planar motions can be found in
  \cite[Ch.~XI]{bottema90}. There, also a more detailed discussion on
  the circularity of trajectories can be found. For planar rational
  curves, the geometric circularity conditions there imply the
  algebraic circularity conditions of \autoref{th:1}.
\end{remark}

As one consequence of \autoref{th:1}, among all rational curves of
fixed degree, curves of high circularity can be generated by rational
motions of low degree. In particular, we obtain the desired bound on
the degree of a rational motion with a given rational trajectory:

\begin{corollary}
  \label{cor:1}
  If a rational curve of circularity $c$ and degree $d$ is a
  trajectory of the rational motion $C = P + \eps Q$, the degree of
  $C$ is not less than $d-c$. If it is of degree $d-c$, the degree defect
  of the spherical motion component equals $d-2c$.
\end{corollary}

We will see below in \autoref{sec:motion-synthesis} that the bound of
\autoref{cor:1} is sharp.

\begin{example}
  Let us illustrate above results by simple examples from
  literature. In case of $\deg C = 1$, the only generic rational
  motions are rotations with fixed axis. Their generic trajectories
  (trajectories that maximise the degree among all trajectories) are
  circles and, of course, entirely circular. Rational motions with
  $\deg C = 2$ are generated by reflecting a moving frame in one
  family of rulings on a quadric $H$, see \cite{hamann11}. The generic
  case is obtained if $H$ is a hyperboloid. In this case, generic
  trajectories are of degree four. Their entirely circularity has
  already been observed in \cite[Ch.\;IX,~\S7]{bottema90}. The
  non-generic trajectories are circles. Points with non-generic
  trajectories lie on two skew lines that coincide if $H$ is a
  hyperboloid of revolution. Finally, if $H$ is a hyperbolic
  paraboloid, generic trajectories are just of degree three and
  circularity one \cite{hamann11}.
\end{example}

\subsection{Motion synthesis}
\label{sec:motion-synthesis}

Now we turn to the task of computing a rational motion
$C = P + \eps Q$ of minimal degree $\deg C = n$ with a given rational
parameterised curve $x = x_0 + x_1\qi + x_2\qj + x_3\qk$ as
trajectory. Denote the degree of the trajectory by $d$ and its
circularity by $c$. By \autoref{cor:1} we have $n \ge d - c$. We will
prove that, up to coordinate changes, exactly one solution of minimal
degree $n = d - c$ exists and we also provide a procedure for its
computation.

\begin{theorem}
  \label{th:2}
  Let $x = x_0 + x_1\qi + x_2\qj + x_3\qk$ be a reduced rational
  parametric equation of a rational curve such that $x(\infty) \peq
  1$. Then there exists a unique monic rational motion polynomial
  $C = P + \eps Q$ of minimal degree, such that $C(\infty) = 1$ and
  the trajectory of $1$ is parameterised by~$x$.
\end{theorem}

A key ingredient in our proof of \autoref{th:2} is a version of the
Euclidean algorithm to compute the \emph{left gcd} of two polynomials
$F,G \in \H[t]$. This is a well-known concept in the theory of
polynomials over rings, see \cite{ore33}. Call the polynomial
$L \in \H[t]$ a \emph{left factor} or \emph{left divisor} of $F$ if
there exists $Q \in \H[t]$ such that $F = LQ$. The polynomial
$D \in \H[t]$ is called \emph{left gcd} of $F$ and $G$, if $D$ is a
left divisor of $F$ and $G$ and any left divisor $E$ of $F$ and $G$
also left divides $D$. The left gcd is unique up to right
multiplication with a non-zero quaternion.

The Euclidean algorithm in this context is based on polynomial
\emph{right} division. Given $R_0,R_1 \in \H[t]$, there exist
$Q_2,R_2 \in \H[t]$ such that $R_0 = R_1Q_2 + R_2$ and
$\deg R_2 < \deg R_1$. Note that the order of factors in the product
$R_1Q_2$ is important. If $D$ is a left divisor of $R_0$ and $R_1$,
then $D$ is also a left divisor of $R_1Q_2$ and of $R_2$. Conversely,
a left divisor of $R_2$ and $R_1$ also left divides $R_0$. Hence the
left gcd of $R_0$ and $R_1$ equals the left gcd of $R_1$ and
$R_2$. Assuming $\deg R_0 \ge \deg R_1$, we also have
$\deg R_0 > \deg R_2$ and we may recursively define sequences
$R_2,R_3,\ldots$ and $Q_2,Q_3,\ldots$ of polynomials with strictly
decreasing degree by
\begin{equation*}
  R_{k-2} = R_{k-1}Q_k + R_k,\quad k \ge 2.
\end{equation*}
The recursion ends as soon as $R_k = 0$ whence the left gcd is
$D = R_{k-1}$. Using polynomial long division, an algorithmic
implementation of this algorithm is straightforward.

\begin{lemma}
  \label{lem:2}
  If $C \in \H[t]$ and $M \in \R[t]$ is an irreducible (over $\R$)
  quadratic factor of $C\cj{C}$, then $M$ and $C$ have a left gcd of
  positive degree.
\end{lemma}
\begin{proof}
  Use polynomial right division to compute $Q,R \in \H[t]$ with $C =
  MQ + R = QM + R$ and $\deg R \le 1$. Then
  \begin{equation*}
    C\cj{C} = (MQ+R)(M\cj{Q}+\cj{R})
            = M(MQ\cj{Q} + Q\cj{R}+R\cj{Q}) + R\cj{R}.
  \end{equation*}
  If $R = 0$, then $M$ itself is a left gcd of $C$ and $M$. Otherwise,
  $M$ is a left factor of $R\cj{R}$ which implies $M \peq
  R\cj{R}$. Hence, $R$ is a left factor of $M$ and also of~$C$.
\end{proof}

\begin{lemma}
  \label{lem:3}
  For any polynomial $C \in \H[t]$ of positive degree and a monic divisor
  $R \in \R[t]$ of $C\cj{C}$ with $\gcd(R,\mrpf C) = 1$, there exists
  a unique monic quaternion polynomial $P$ (the left gcd of $C$ and
  $R$) and a unique quaternion polynomial $Q$ such that $PQ=C$ and
  $P\cj{P}=R$.
\end{lemma}
\begin{proof}
  The polynomial $R$ has no real linear factor because such a factor
  necessarily is a divisor of $C$ which contradicts
  $\gcd(R, \mrpf C) = 1$.  Using the Euclidean algorithm, we can
  compute the unique monic left gcd $P_0$ of $C$ and $R$, i.e.,
  \begin{equation}
    \label{eq:4}
    P_0Q=C,\quad P_0P_1=R
  \end{equation}
  with $Q,P_1 \in \H[t]$. We claim that $P_0$ and $P_1$ both have no
  real polynomial factor:
  \begin{itemize}
  \item A real polynomial factor of $P_0$ is also a real polynomial
    factor of $C$ and $R$ and contradicts $\gcd(R, \mrpf C) = 1$.
  \item A real polynomial factor of $P_1$ must have a real quadratic
    factor $M$ which is irreducible over $\R$ (because $R$ has no real
    linear factor). That is, for some $P'_1 \in \H[t]$ we have
    $P_1=P'_1M$. We can write $R=P_0P_1=P_0P'_1M=R'M$ where $R'$ is
    real. From this we infer $\cj{P_0}R'=\cj{P_0}P_0P'_1$. Because
    $\cj{P_0}$ has no real polynomial factor, the real polynomial
    $\cj{P_0}P_0$ divides $R'$ and $\cj{P_0}P_0M$ divides $R$ but also
    $C\cj{C}=Q\cj{Q}P_0\cj{P_0}$ (a multiple of~$R$). This implies
    that $M$ divides $Q\cj{Q}$ and, by \autoref{lem:2}, $Q$ and $M$
    have a left gcd $P_2$ with $\deg P_2 > 0$. But then, by
    \eqref{eq:4}, $P_0P_2$ is a left divisor of $C$ and $R$ and
    $\deg P_0P_2 > \deg P_2$. This contradicts the fact that $P_0$ is
    left gcd of $C$ and~$R$.
  \end{itemize}
  Now we claim that $P_1=\cj{P_0}$. Left multiplying the second
  equation in \eqref{eq:4} with $\cj{P_0}$ we obtain
  $\cj{P_0}R=\cj{P_0}P_0P_1$. We have a real polynomial factor $R$ on
  the left and a real polynomial factor $\cj{P_0}P_0$ on the
  right. Because neither $P_0$ nor $P_1$ have real polynomial factors,
  we have $R = \cj{P_0}P_0$ and $P_1=\cj{P_0}$ follows. Thus, with
  $P \coloneqq P_0$, we have existence.

  As to uniqueness, we observe that $P$ is necessarily the unique
  monic left gcd of $C$ and $R$. Uniqueness of $Q$ follows from
  uniqueness of~$P$.
\end{proof}

\begin{proof}[Proof of \autoref{th:2}]
  Denote the degree of the rational parametric equation $x$ by $d$ and
  its circularity by $c$. The condition $x(\infty) \peq 1$ implies
  $\deg x_0 > \deg x_i$ for $i=1,2,3$ and it is no loss of generality
  to assume that $x_0$ is monic. There exists a monic polynomial $g$
  of degree $2c$ and two relatively prime polynomials $w$ (monic),
  $y \in \R[t]$ with $\deg w = d - 2c$, $\deg y < 2(d - c)$, such that
  \begin{equation}
    \label{eq:5}
    x_0 = g w,\quad x_1^2 + x_2^2 + x_3^2 = gy.
  \end{equation}
  Denote by $C \coloneqq P + \eps Q$ a (yet unknown) motion polynomial
  of minimal degree that parameterises the sought rational motion. The
  parametric trajectory of $1$ equals $x$ if
  \begin{equation}
    \label{eq:6}
    x_0 + x_1\qi + x_2\qj + x_3\qk \peq P\cj{P} + 2P\cj{Q}
    \quad\text{where}\quad
    P\cj{Q} + Q\cj{P} = 0.
  \end{equation}
  With $D \coloneqq x_1\qi + x_2\qj + x_3\qk$, we have
  \begin{equation}
    \label{eq:7}
    D\cj{D} = x_1^2 + x_2^2 + x_3^2 = gy
    \quad\text{and}\quad
    wD\cj{D} = wgy = x_0y.
  \end{equation}
  By \autoref{lem:3}, there exist polynomials $P_0 \in \H[t]$ and
  $Q_0 \in \H[t]$ such that
  \begin{equation*}
    P_0\cj{P_0} = g, \quad
    P_0\cj{Q_0} = D.    
  \end{equation*}
  With $P \coloneqq wP_0$ and $Q \coloneqq Q_0/2$ we then have
  \begin{equation*}
    P\cj{P} = w^2P_0\cj{P_0} = w^2 g = wx_0, \quad
    2P\cj{Q} = wP_0\cj{Q_0}=wD = w(x_1\qi + x_2\qj + x_3\qk).
  \end{equation*}
  This shows that $P$ and $Q$ solve the first of the two equations in
  \eqref{eq:6}. The second equation follows from the vanishing of the
  scalar part of $D$:
  $0 = D + \cj{D} = wD + w\cj{D} = 2(P\cj{Q} + Q\cj{P})$. Thus, the
  polynomial $C = P + \eps Q$ with $P$, $Q$ computed as above is one
  solution to our problem and it is of minimal
  degree~$n \coloneqq d - c$.
  
  Conversely, assume that $C'=P'+\eps Q'$ is another solution of
  minimal degree $\deg C' = n$. Again, we have $\deg P' > \deg Q'$ and
  $P'$ is monic. By the first of the two equations in \eqref{eq:6} we
  get
  \begin{equation}
    \label{eq:8}
    P'\cj{P'} = w'x_0, \quad \quad
    2P'\cj{Q'} = w'D = w'(x_1\qi + x_2\qj + x_3\qk),
  \end{equation}
  for some monic real polynomial $w'$ of degree $d-2c$. By
  \eqref{eq:5}, \eqref{eq:7}, and \eqref{eq:8} we get
  \begin{equation*}
    gy=D\cj{D}=4\frac{P'\cj{P'}Q'\cj{Q'}}{w'^2}=4\frac{w'x_0Q'\cj{Q'}}{w'^2}
    =4\frac{x_0}{w'}Q'\cj{Q'}=4\frac{gw}{w'}Q'\cj{Q'}.
  \end{equation*}
  If we divide both sides by $g$ and multiply with $w'$, we get
  \begin{equation*}
    yw'=4wQ'\cj{Q'}.
  \end{equation*}
  As $w$ and $y$ are relatively prime, we get that $w$ divides
  $w'$. But $\deg w = \deg w'$ and $w$ and $w'$ are monic. Thus, we have
  \begin{equation}
    \label{eq:9}
    w=w'\quad \text{and} \quad y=4Q'\cj{Q'}.
  \end{equation}
  Left-multiplying the second of two equations in \eqref{eq:8} by
  $Q'$ and using \eqref{eq:9}, we get
  \begin{equation*}
    P'=w'\frac{DQ'}{2Q'\cj{Q'}}=w\frac{2DQ'}{y}.
  \end{equation*}
  As $w$ and $y$ are relatively prime and $w$ is monic,
  $P'_0 \coloneqq 2DQ'/y$ must be a monic polynomial. We have
  \begin{equation*}
    P'_0\cj{P'_0}=\frac{4Q'\cj{Q'}D\cj{D}}{y^2}=\frac{D\cj{D}}{y}=g, \quad 2P'_0\cj{Q'}=\frac{4Q'\cj{Q'}D}{y}=D.
  \end{equation*}
  Because of $\deg P'_0 = \deg P_0$ and because the left gcd of $g$ and
  $D$ is unique, this implies $P_0 = P'_0$ and $Q' = Q_0/2 = Q$. This
  concludes the proof of uniqueness.
\end{proof}

\begin{remark}
  In case of a rational curve of circularity zero, the primal part $P$
  of the rational motion $C = P + \eps Q$ is a real polynomial and the
  trivial solution consists of the curvilinear translation along the
  given curve. In all other cases, \autoref{th:2} guarantees solutions
  of lesser degree than curvilinear translations.
\end{remark}

Our formulation of \autoref{th:2} requires a special moving point, the
origin $1$, to generate the rational curve and produces a motion
polynomial $C$ with the special property $C(\infty) \peq 1$.  These
conditions are just coordinate dependencies and are convenient for our
algebraic formulation and proof. If we allow coordinate changes, we
may translate $1$ to an arbitrary moving point (this amounts to
replacing $C$ by $CT^{-1}$ with a suitable translation, as in the
proof of \autoref{th:1}) and orient the coordinate axis arbitrarily
(this amounts to right-multiplying $C$ with $r \in \H$). Thus, we can
re-formulate \autoref{th:2} as

\begin{corollary}
  \label{cor:2}
  There is a unique (up to coordinate changes) rational motion of
  minimal degree in the dual quaternion model of rigid body
  displacements with a prescribed rational trajectory. If the
  trajectory is of degree $d$ and circularity $c$, the minimal motion
  is of degree $n = d - c$.
\end{corollary}

\begin{remark}
  Frequently, the degree of a rational motion is defined as maximal
  degree of one of its trajectories. With this concept of degree,
  \autoref{cor:2} is no longer true. A counter example is a circle
  which occurs as trajectory of a circular translation and of a
  rotation about its centre. For both motions the maximal degree of a
  trajectory is two.
\end{remark}

Let us illustrate \autoref{th:2} by two examples:

\begin{example}
  The rational curve given by the parameterised equations
  \begin{gather*}
    x_0 = (t^2+2 t+2) (t^2+2 t+5) (t+1),\quad
    x_1 = 2  (t^2-5) (t^2+2 t+2),\\
    x_2 = -4  (t+5) (t+1)^2,\quad
    x_3 = -2 t (t+2) (t+5) (t+1)
  \end{gather*}
  is of degree five. From the factorisation
  \begin{equation*}
    x_1^2 + x_2^2 + x_3^2 = 8(t^2+2t+5)(t^2+4t+5)(t^2+2t+2)^2
      \end{equation*}
  we see that its circularity is two. By \autoref{th:2}, it is
  trajectory of a rational motion of minimal degree three whose primal
  part, by \autoref{th:1}, has precisely one linear factor. With $y$,
  $w$ and $g$ defined by
  \begin{equation*}
    g = (t^2+2 t+2) (t^2+2 t+5),\quad
    w = t + 1,\quad 
    y = 8(t^2+4 t+5) (t^2+2 t+2),   
  \end{equation*}
  we have
  \begin{equation*}
    x_0 = g w,\quad x_1^2+x_2^2+x_3^2=gy.
  \end{equation*}
  Setting $D = x_1\qi + x_2\qj + x_3\qk$ and using the Euclidean
  algorithm we compute, according to \autoref{lem:3}, the left gcd
  $P_0$ of $D$ and $g$: With
  \begin{equation*}
    P_0 = t^2 + 2t +1 - (t+1)\qi+ (2t+2)\qj-2\qk
  \end{equation*}
  and 
  \begin{equation*}
    Q_0 = -2t-2-(2t^2+4t+2)\qi+(2t+6)\qj+(2t^2+8t+6)\qk 
  \end{equation*}
  we have $g = P_0\cj{P_0}$, $D = P_0\cj{Q_0}$. Then we set $P = wP_0$
  and $Q = Q_0/2$.  As expected, $P$ has a real polynomial factor of
  degree $1$.  The motion polynomial we want is $C \coloneqq P+\eps Q$.
\end{example}

\begin{example}
  \label{ex:3}
  The Viviani curve \cite{peternell13} is given by
  \begin{equation*}
    x_0 = (1+t^2)^2,\quad
    x_1 = (1-t)^2(1+t)^2,\quad
    x_2 = 2t(1-t)(1+t),\quad
    x_3 = 2t(1+t^2).
  \end{equation*}
  In order to meet the requirements of \autoref{th:2}, we translate it
  by the vector $(-1, 0, 0)$ and obtain
  \begin{equation}
    \label{eq:10}
    x_0 = (1+t^2)^2,\quad
    x_1 = -4t^2,\quad
    x_2 = 2t(1-t)(1+t),\quad
    x_3 = 2t(1+t^2).
  \end{equation}
  This curve lies on the unit sphere and is entirely circular. As in
  the previous example, we compute
  \begin{equation*}
    g = \gcd(x_0,x_1^2+x_2^2+x_3^2) = (1+t^2)^2,\quad
    w = x_0/g = 1,\quad
    y = (x_1^2+x_2^2+x_3^2)/g = 8t^2.
  \end{equation*}
  The left gcd of $D \coloneqq x_1\qi + x_2\qj + x_3\qk$ and $g$ is
  \begin{equation*}
    P_0 = t^2 - t(\qj + \qk) - \qi,
  \end{equation*}
  the right quotient of $D$ and $P_0$ is
  \begin{equation*}
    Q_0 = 2t(\qk - \qj),
  \end{equation*}
  and the minimal motion to the curve \eqref{eq:10} is
  \begin{equation*}
    C = t^2 - (\qj + \qk - \eps(\qj + \qk))t - \qi.
  \end{equation*}
  We can simplify by conjugating with $T \coloneqq 1 +
  \frac{1}{2}\eps\qi$ whence we get
  \begin{equation*}
    C' \coloneqq \cj{T}CT = t^2 - t(\qj + \qk) - \qi = (t - \qk)(t - \qj).
  \end{equation*}
  This shows that $C'$ is the composition of two rotations about the
  second coordinate axis and the third coordinate axis with equal
  angular velocities, i.e., the motion generated by the rolling of a
  spherical circle of radius $\pi/4$ on a spherical circle of the same
  radius. This is illustrated in \autoref{fig:viviani}. There, the
  moving frame is rigidly attached to the rolling circle.
\end{example}

\begin{figure}
  \centering
  \includegraphics{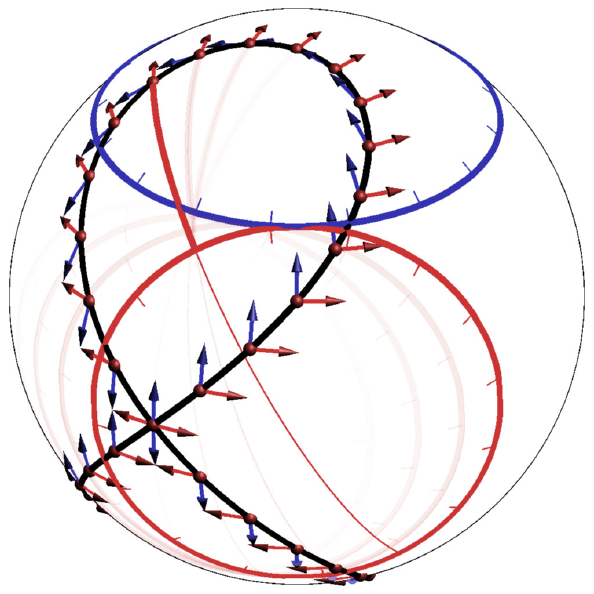}
  \caption{Minimal motion of the Viviani curve (rolling of spherical circles)}
  \label{fig:viviani}
\end{figure}

\section{Discussion of results}
\label{sec:discussion}

This article unveiled some relations between rational motions and
their trajectories. A rational curve occurs as trajectory of a unique
(up to coordinate changes) rational motion of minimal degree in the
dual quaternion model. This was a surprise to the authors as a mere
trajectory seems to leave a lot of freedom for the construction of a
suitable rational motion. Apparently, the requirement for minimality
is rather restrictive.

Calling a rational curve ``generic'' if its circularity is zero and a
rational motion ``generic'' if its primal part has no real factors
(the spherical motion component has full degree), we may also say that
rational motions of minimal quaternion degree to a generic trajectory
are non-generic. Conversely, the trajectories of a generic rational
motion are entirely circular and hence non generic.

In conjunction with the factorisation of generic rational motions
\cite{hegedus13} and its extension \cite{FMP} to non-generic motion
polynomials, our results contribute to a variant of Kempe's
Universality Theorem \cite[Section~3.2]{demaine07} for rational space
curves. Via motion factorisation, it is possible to construct linkages
to generate a given rational motion and hence also a given rational
trajectory. The dual quaternion degree of the generating motion is
directly related to the number of links and joints in the mechanism. As a
consequence, rational curves can be generated by spatial linkages with
much fewer links and joints than those implied by the asymptotic
bounds for algebraic space curves given by \cite{abbott08}. For
circular curves, the numbers of links and joints are even less. A
precise formulation and a rigorous proof will be worked out in a
future paper.

Our main result in Corollary~\ref{cor:2} raises questions. Is a similar
uniqueness statement true for algebraic curves and algebraic motions?
What rational motions are not minimal for any of their trajectories?
The Darboux motion and the circular translation of Remark~\ref{rem:1} are
examples. Finally, determine simple (classes of) curves with simple
minimal motion as in Example~\ref{ex:3} and exploit their low degree and
rationality in a CAGD or engineering context.

\section*{Acknowledgements}
\label{sec:acknowledgements}

This work was supported by the Austrian Science Fund (FWF): P~26607
(Algebraic Methods in Kinematics: Motion Factorisation and Bond
Theory).

\bibliographystyle{amsplain}
\providecommand{\bysame}{\leavevmode\hbox to3em{\hrulefill}\thinspace}
\providecommand{\MR}{\relax\ifhmode\unskip\space\fi MR }
\providecommand{\MRhref}[2]{  \href{http://www.ams.org/mathscinet-getitem?mr=#1}{#2}
}
\providecommand{\href}[2]{#2}

\end{document}